\documentclass[12pt]{amsart}
\usepackage{format}

\title{Regular Functions on the $K$-nilpotent cone}

\begin{document}

\maketitle

\begin{abstract}
Let $G$ be a complex reductive algebraic group with Lie algebra $\fg$ and let $G_{\RR}$ be a real form of $G$ with maximal compact subgroup $K_{\RR}$. Associated to $G_{\RR}$ is a $K \times \CC^{\times}$-invariant subvariety $\cN_{\theta}$ of the (usual) nilpotent cone $\cN \subset \fg^*$. In this article, we will derive a formula for the ring of regular functions $\CC[\cN_{\theta}]$ as a representation of $K \times \CC^{\times}$.

Some motivation comes from Hodge theory. In \cite{SchmidVilonen2011}, Schmid and Vilonen use ideas from Saito's theory of mixed Hodge modules to define canonical good filtrations on many Harish-Chandra modules (including all standard and irreducible Harish-Chandra modules). Using these filtrations, they formulate a conjectural description of the unitary dual. If $G_{\RR}$ is split, and $X$ is the spherical principal series representation of infinitesimal character $0$, then conjecturally $\gr(X) \simeq \CC[\cN_{\theta}]$ as representations of $K \times \CC^{\times}$. So a formula for $\CC[\cN_{\theta}]$ is an essential ingredient for computing Hodge filtrations.
\end{abstract}

\section{Introduction}

Let $G$ be a complex connected reductive algebraic group and let $G_{\RR}$ be a real form of $G$. Choose a maximal compact subgroup $K_{\RR}$ of $G_{\RR}$ and let $\theta: G \to G$ be the corresponding Cartan involution. Let $K$ be the group of $\theta$-fixed points (i.e. the complexification of $K_{\RR}$). Note that $K$ is a complex reductive algebraic group (it is often disconnected). Write $\fg$,$\fk$ for the Lie algebras and let $\fp = \fg^{-d\theta}$. There is a Cartan decomposition
\begin{equation}\label{eq:Cartan}\fg = \mathfrak{k} \oplus \fp\end{equation}
Using (\ref{eq:Cartan}), we can identify $\fp^*$ with $(\fg/\fk)^*$, a subspace of $\fg^*$. 

Let $G \times \CC^{\times}$ act on $\fg^*$ by the usual formula
$$(g,z) \cdot \zeta = z\Ad^*(g)\zeta, \qquad g \in \widetilde{G}, \ z \in \CC^{\times}, \ \zeta \in \fg^*.$$
Note that $\fp^*$ is stable under $K \times \CC^{\times} \subset G \times \CC^{\times}$. Let $\cN$ denote the nilpotent cone in $\fg^*$. Recall that $G \times \CC^{\times}$ acts on $\cN$ (with finitely many orbits). The $K$-\emph{nilpotent cone} is the $K \times \CC^{\times}$-invariant subvariety
$$\cN_{\theta} = \cN \cap \fp^* \subset \fp^*$$
This subvariety (and the $K \times \CC^{\times}$-action on it) is closely related to the representation theory of $G_{\RR}$, see \cite{Vogan1991}. The main result of this paper is an explicit description of the ring of regular functions $\CC[\cN_{\theta}]$ as a representation of $K \times \CC^{\times}$ (in the case when $G_{\RR}$ is split modulo center). Some motivation for this problem will be discussed at the end of Section \ref{sec:Khat}.

Since $K$ is (in general) a disconnected group, its irreducible representations cannot be easily parameterized using the theory of highest weights. In \cite{Vogan2007}, Vogan gives a parameterization of a very different flavor (making essential use of the fact that $K$ is a symmetric subgroup of $G$). We will recall some of the details in Section \ref{sec:Khat}. 

There are several well-known results regarding the structure of $\CC[\cN_{\theta}]$ as a $K$-representation. In \cite{KostantRallis1971}, it is shown that for $G_{\RR}$ quasi-split, $\CC[\cN_{\theta}]$ is isomorphic as a $K$-representation to the induced representation $\Ind^K_M \mathrm{triv}$, where $M$ is centralizer in $K$ of a maximal abelian subspace $\mathfrak{a} \subset \mathfrak{p}$. Equivalently, $\CC[\cN_{\theta}]$ is isomorphic as a $K$-representation to a spherical principal series representation of $G_{\RR}$. Importantly, these results do not provide any information about the $\CC^{\times}$-action on $\CC[\cN_{\theta}]$. To understand this structure, we must adopt a different approach. First, we relate the ring of regular functions $\CC[\cN_{\theta}]$ (regarded as a representation of $K \times \CC^{\times}$) to the ring of regular functions $\CC[\cN]$ (regarded as a representation of $G \times \CC^{\times})$. For $G_{\RR}$ split modulo center, we prove the following formula in Corollary \ref{cor:formulasplit}
\begin{equation}\label{eq:mainformula}\CC[\cN_{\theta}]|_{K \times \CC^{\times}} = \CC[\cN]|_{K \times \CC^{\times}} \otimes [\wedge(\fk)]\end{equation}
Here $[\wedge(\fk)]$ denotes the signed graded exterior algebra associated to $\fk$ (this formula takes place in the Grothendieck group of admissible representations of $K \times \CC^{\times}$). The proof of this result is not purely formal---we make essential use of the fact that $\cN$ is Cohen-Macaulay and that $\cN_{\theta} \subset \cN$ is a complete intersection of codimension $\dim(\fk)$ (for the latter assertion, we use that $G_{\RR}$ is split modulo center). The structure of $\CC[\cN]$ as a $G \times \CC^{\times}$-representation is well-known (it can be computed using Lusztig's $q$-analog of Kostant's partition function, see \cite{Mcgovern1989}). We then `restrict' this description to $K \times \CC^{\times}$ to obtain a formula for $\CC[\cN_{\theta}]$. The final result is Theorem \ref{thm:branching}.

Along the way, we introduce a restriction map on equivariant K-theory
$$K^{G}(\cN) \to K^{K}(\cN_{\theta})$$
(this map is defined and studied in Section \ref{sec:Ktheory}). In fact, this map arises as the `associated graded' of a map from representations of $G$ (regarded as a real group) to representations of $G_{\RR}$ (see Remark \ref{rmk:GGR} for more details). We will pursue this point in future work.  

\subsection{Notation}

Let $R$ be an algebraic group. An algebraic $R$-representation $V$ is \emph{admissible} if every irreducible $R$-representation appears in $V$ with finite multiplicity. Consider the abelian categories
\begin{align*}
    \Rep(R) &= \text{algebraic representations of } R\\
    \Rep_a(R) &= \text{admissible algebraic representations of } R\\
    \Rep_f(R) &= \text{finite-dimensional algebraic representations of } R
\end{align*}
There are obvious embeddings $\Rep_f(R) \subset \Rep_a(R) \subset \Rep(R)$. 

Now let $\widetilde{R} = R \times \CC^{\times}$. Then $\Rep(\widetilde{R})$, $\Rep_a(\widetilde{R})$, and $\Rep_f(\widetilde{R})$ can be defined as above. We will also consider the category
\begin{align*}
\Rep_{aa}(\widetilde{R}) = &\text{algebraic representations of } \widetilde{R} \text{ which are admissible}\\
&\text{as representations of both } R \text{ and } \CC^{\times}\end{align*}
There are obvious embeddings $\mathrm{Rep}_f(\widetilde{R}) \subset \mathrm{Rep}_{aa}(\widetilde{R}) \subset \mathrm{Rep}_a(\widetilde{R}) \subset \mathrm{Rep}(\widetilde{R})$. We will denote the Grothendieck groups by $K(R),K_a(R),K_f(R),K_{aa}(\widetilde{R})$, and so on. Write $\mathrm{Irr}(R)$ for the set of (equivalence classes of) irreducible representations of $R$. Then 
$$\mathrm{Irr}(\widetilde{R}) = \{\tau q^n \mid \tau \in \mathrm{Irr}(R), \ n \in \ZZ\},$$
where $q^n$ denotes the degree-$n$ character of $\CC^{\times}$ and $\tau q^n$ is shorthand for the irreducible  $\widetilde{R}$-representation $\tau \otimes q^n$. The group $K_f(R)$ (resp. $K_f(\widetilde{R})$) can be identified with (finite) integer combinations of $\tau \in \mathrm{Irr}(R)$ (resp. $\tau q^n \in \mathrm{Irr}(\widetilde{R})$). The group $K_a(R)$ (resp. $K_a(\widetilde{R})$) can be identified with \emph{formal} integer combinations of $\tau \in \mathrm{Irr}(R)$ (resp. $\tau q^n \in \mathrm{Irr}(\widetilde{R})$). Finally, $K_{aa}(\widetilde{R})$ can be identified with formal integer combinations of $\tau q^n$ such that for each $\tau \in \mathrm{Irr}(R)$ and for each $n \in \ZZ$ only finitely many $\tau q^n$ appear with nonzero multiplicity. The tensor product of representations turns $K_f(R)$ into a commutative ring, $K_a(R)$ into a $K_f(R)$-module, and $K_{aa}(\widetilde{R})$ into a $K_f(\widetilde{R})$-module. 

If $X$ is a scheme equipped with an algebraic $R$-action, we write $\Coh^R(X)$ for the category of (strongly) $R$-equivariant coherent sheaves on $X$ and $K^R(X)$ for its Grothendieck group.

\subsection{Acknowledgments}

I would like to thank David Vogan and Yu Zhao for many helpful discussions. 

\section{Irreducible representations of $K$}\label{sec:Khat}

In this section, we will recall a parameterization of $\mathrm{Irr}(K)$ due to Vogan (\cite{Vogan2007}).

Suppose $H$ is a $\theta$-stable maximal torus in $G$. Write $\Delta = \Delta(G,H)$ for the roots of $H$ on $\fg$. Since $H$ is $\theta$-stable, $\theta$ acts on $\Delta$ by $\alpha \mapsto \alpha \circ \theta$. A root $\alpha \in \Delta$ is imaginary (resp. real, resp. complex) if $\theta\alpha = \alpha$ (resp. $\theta\alpha = -\alpha$, resp. $\theta\alpha \notin \{\pm \alpha\}$). If $\alpha$ is imaginary, then $\theta$ acts on the root space $\fg_{\theta}$ by $\pm \mathrm{Id}$. We say that $\alpha$ is compact (resp. noncompact) if $\theta|_{\fg_{\alpha}} = \mathrm{Id}$ (resp. $\theta|_{\fg_{\alpha}} = -\mathrm{Id}$). So $\Delta$ is partitioned
$$\Delta = \Delta_{i\RR} \sqcup \Delta_{\RR} \sqcup \Delta_{\CC}$$
into imaginary, real, and complex roots, and $\Delta_{i\RR}$ is partitioned
$$\Delta_{i\RR} = \Delta_c \sqcup \Delta_n$$
into compact and noncompact roots. Note that $\Delta_{i\RR}$ is a root system. If we choose a positive system $\Phi^+ \subset \Delta_{i\RR}$, we can define the element $\rho_{i\RR} = \frac{1}{2}\sum \Phi^+ \in \fh^*$.

\begin{definition}[Sec 6, \cite{ATLV}]\label{def:Langlands}
A continued Langlands parameter for $(G,K)$ is a triple $(H,\gamma,\Phi^+)$ where
\begin{enumerate}
    \item $H$ is a $\theta$-stable maximal torus in $G$.
    \item $\gamma$ is a formal sum $\gamma_0+\rho_{i\RR}$, where $\gamma_0$ is a one-dimensional $(\fh,H^{\theta})$-module (we define $d\gamma = d\gamma_0 + \rho_{i\RR}$).
    \item $\Phi^+$ is a positive system for $\Delta_{i\RR}$.
\end{enumerate}
There is a $K$-action on the set of continued Langlands parameters. Two parameters are \emph{equivalent} if they are conjugate under $K$. A continued Langlands parameter is \emph{standard} if
\begin{itemize}
    \item[(4)] For every $\alpha \in \Phi^+$, $\langle d\gamma, \alpha^{\vee}\rangle \geq 0$.
\end{itemize}
A standard Langlands parameter is \emph{nonzero} if
\begin{itemize}
    \item[(5)] For every $\alpha \in \Phi^+$ which is simple and compact, $\langle d\gamma,\alpha^{\vee}\rangle \neq 0$. 
\end{itemize}
A nonzero Langlands parameter is \emph{final} if
\begin{itemize}
    \item[(6)] If $\alpha \in \Delta_{\RR}$ and $\langle d\gamma,\alpha^{\vee}\rangle =0$, then $\alpha$ does not satisfy the Speh-Vogan parity condition (\cite{SpehVogan}).
\end{itemize}
Write $\mathcal{P}_L(G,K)$ for the set of equivalence classes of final Langlands parameters.
\end{definition}

To the equivalence class of a continued Langlands parameter $\Gamma$, one can associate a virtual Harish-Chandra module $[I(\Gamma)]$. If $\Gamma$ is standard, then $[I(\Gamma)]$ is represented by a distinguished $(\fg,K)$-module $I(\Gamma)$. If $\Gamma$ is nonzero, then $I(\Gamma)\neq 0$. If $\Gamma$ is final, there is a unique irreducible quotient $I(\Gamma) \twoheadrightarrow J(\Gamma)$. The following is a version of the Langlands classification.

\begin{theorem}[Thm 6.1, \cite{ATLV}]
The map $\Gamma \mapsto J(\Gamma)$ defines a bijection between $\mathcal{P}_L(G,K)$ and isomorphism classes of irreducible $(\fg,K)$-modules. 
\end{theorem}

For any continued Langlands parameter $\Gamma$, the infinitesimal character of the virtual $(\fg,K)$-module $[I(\Gamma)]$ corresponds, under the Harish-Chandra isomorphism, to the $W$-orbit of $d\gamma$. If $\Gamma$ is final, then $I(\Gamma)$ is tempered if and only if the restriction of $d\gamma$ to $\fh^{-d\theta}$ lies in the imaginary span of $\Delta$ (i.e. $d\gamma|_{\fh^{-d\theta}}$ is imaginary). In this case, $I(\Gamma)$ is irreducible, i.e. $I(\Gamma)=J(\Gamma)$. 

\begin{definition}
A final Langlands parameter $(H,\gamma,\Phi^+)$ is \emph{tempered with real infinitesimal character} if $d\gamma|_{\fh^{-d\theta}} = 0$. Write $\mathcal{P}_L^{t,\RR}(G,K)$ for the set of equivalence classes of such parameters.
\end{definition}

\begin{cor}
The map $\Gamma \mapsto J(\Gamma)$ defines a bijection between $\mathcal{P}_L^{t,\RR}(G,K)$ and isomorphism classes of irreducible tempered $(\fg,K)$-modules with real infinitesimal character.
\end{cor}

Suppose $\mu$ is an irreducible representation of $K$. Choose a maximal torus $T \subset K$ and a positive system $\Phi_c^+$ for $\Delta(K,T)$. Let $\rho_c = \frac{1}{2}\sum \Phi_c^+ \in \mathfrak{t}^*$. Let $\lambda \in \mathfrak{t}^*$ be a highest weight of $\mu$ (if $K$ is disconnected, $\lambda$ need not be unique). The $K$-norm of $\mu$ is defined by the formula
$$|\mu|^2 := \langle \lambda+2\rho_c,\lambda+2\rho_c\rangle$$
It is not hard to see that $|\mu|$ is independent of $\Phi^+_c$ and $\lambda$.

We say that $\mu$ is a \emph{lowest} $K$\emph{-type} in a a $(\fg,K)$-module $X$ if $\mu$ appears in $X$ with nonzero multiplicity and $|\mu|$ is minimal among all irreducible representations of $K$ with this property.

\begin{theorem}[Thm 11.9, \cite{Vogan2007}]\label{thm:LKTs}
The following are true:
\begin{itemize}
    \item[(i)] If $\Gamma \in \mathcal{P}_L^{t,\RR}(G,K)$, then $I(\Gamma)$ contains a unique lowest $K$-type $\mu(\Gamma)$. 
    \item[(ii)]The map $\Gamma \mapsto \mu(\Gamma)$ defines a bijection between $\mathcal{P}_L^{t,\RR}(G,K)$ and $\mathrm{Irr}(K)$.
    \item[(iii)] There is a total order on $\mathcal{P}_L^{t,\RR}(G,K)$ such that the (infinite) square matrix $m(\Gamma,\Gamma')$ defined by the formula
    $$[I(\Gamma)] = \sum_{\Gamma' \in \mathcal{P}_L^{t,\RR}} m(\Gamma,\Gamma')\mu(\Gamma')$$
    is upper triangular with $1$'s along the diagonal. 
    \item[(iv)] In particular, this matrix $m(\Gamma,\Gamma')$ is invertible. Write $M(\Gamma,\Gamma')$ for its inverse (which is also upper triangular).
    \item[(v)] The entries of the matrices $m(\Gamma,\Gamma')$ and $M(\Gamma,\Gamma')$ can be computed by an algorithm.
\end{itemize}
\end{theorem}
The algorithm in (v) is described in \cite{Vogan2007}. 

Now suppose that $G_{\RR}$ is split modulo center. This means that there is a maximal torus $H_s$ in $G$ such that $\Delta(G,H_s)=\Delta_{\RR}(G,H_s)$ (and so $\Delta_{i\RR}(G,H_s) = \emptyset$). Let $\Gamma_0$ be the parameter
$$\Gamma_0 := (H_s,0,\emptyset) \in \mathcal{P}_L^{t,\RR}(G,K)$$
By a result of Kostant (\cite{kostant1969}) there is an identity in $K_a(K)$
$$\CC[\cN_{\theta}]|_K = I(\Gamma_0)|_K$$
So by Theorem \ref{thm:LKTs}, we can write $\CC[\cN_{\theta}]$ as a formal integer sum of the irreducible $K$-representations $\mu(\Gamma)$ (this idea has been implemented in the atlas software). This is quite useful information, but it does \emph{not} give us the grading on $\CC[\cN_{\theta}]$, which is part of what we're after.

In \cite{SchmidVilonen2011}, Schmid and Vilonen define canonical good filtrations on all Harish-Chandra modules of `functorial origin', including all standard and irreducible Harish-Chandra modules. These canonical good filtrations (and their associated gradeds) should be closely related to questions of unitarity. 

The associated graded of a Harish-Chandra module with respect to a good filtration can be regarded as a representation of $\widetilde{K}$ (in fact, as a class in $K_{aa}(\widetilde{K})$). It is conjectured in \cite{SchmidVilonen2011} that
$$\CC[\cN_{\theta}]|_{\widetilde{K}} = \gr I(\Gamma_0)|_{\widetilde{K}}$$
It is also suggested that the Hodge filtration on an \emph{arbitrary} standard module (of an arbitrary group) can be reduced to this case (via cohomological induction and a deformation argument). So computing the class $\CC[\cN_{\theta}]|_{\widetilde{K}}$ in the case when $G_{\RR}$ is split is central to the program of computing Hodge fitlrations. In Theorem \ref{thm:branching}, we will give a formula for $\CC[\cN_{\theta}]|_{\widetilde{K}}$ in terms of the classes $I(\Gamma)q^n$.

\section{A restriction map $K^{\widetilde{G}}(\cN) \to K^{\widetilde{K}}(\cN_{\theta})$}\label{sec:Ktheory}

In this section, we will define a restriction map
$$K^{\widetilde{G}}(\cN) \to K^{\widetilde{K}}(\cN_{\theta})$$
Since $\cN$ and $\cN_{\theta}$ are singular, we cannot proceed directly. Instead, we follow the standard approach outlined (for example) in \cite[Sec 5.3]{Chriss-Ginzburg}: we first regard $\cN$ (resp. $\cN_{\theta}$) as a subvariety of $\fg^*$ (resp. $\fp^*$) and then apply the restriction map $K^{\widetilde{G}}(\fg^*) \to K^{\widetilde{K}}(\fp^*)$ (defined in the usual way, as an alternating sum of Tor functors).

Our first proposition describes the relationship between $K^{\widetilde{G}}(\cN)$, $K^{\widetilde{G}}(\fg^*)$ $K_{aa}(\widetilde{G})$, and $K_f(\widetilde{G})$.

\begin{prop}\label{prop:g*facts}
The following are true:
\begin{itemize}
    \item[(i)] If $\mathcal{E} \in \Coh^{\widetilde{G}}(\cN)$, then $\Gamma(\cN,\mathcal{E})|_{\widetilde{G}} \in \Rep_{aa}(\widetilde{G})$. This defines an exact functor $\Coh^{\widetilde{G}}(\cN) \to \Rep_{aa}(\widetilde{G})$, and hence a group homomorphism
    $$\Gamma(\bullet)|_{\widetilde{G}}: K^{\widetilde{G}}(\cN) \to K_{aa}(\widetilde{G})$$
    \item[(ii)] The homomorphism in (i) is injective.
    \item[(iii)] If $\mathcal{E} \in \Coh^{\widetilde{G}}(\fg^*)$, then $\Gamma(\fg^*,\mathcal{E})|_{\widetilde{G}} \in \Rep_a(\widetilde{G})$. This defines an exact functor $\Coh^{\widetilde{G}}(\fg^*) \to \Rep_a(\widetilde{G})$, and hence a group homomorphism
$$\Gamma(\bullet)|_{\widetilde{G}}: K^{\widetilde{G}}(\fg^*) \to K_a(\widetilde{G})$$
    \item[(iv)] Restriction along $\{0\} \subset \fg^*$ 
    induces an exact functor $\Coh^{\widetilde{G}}(\fg^*) \to \Coh^{\widetilde{G}}(\{0\}) \simeq \Rep_f(\widetilde{G})$, which in turn induces a group isomorphism
    $$|_{\{0\}}: K^{\widetilde{G}}(\fg^*) \xrightarrow{\sim} K_f(\widetilde{G})$$
    \item[(v)] The direct image along the closed embedding $j: \cN \hookrightarrow \fg^*$ induces an exact functor $\Coh^{\widetilde{G}}(\cN) \to \Coh^{\widetilde{G}}(\fg^*)$, and therefore a group homomorphism
    $$j_*: K^{\widetilde{G}}(\cN) \to K^{\widetilde{G}}(\fg^*)$$
    \item[(vi)] If $V \in \Rep_f(\widetilde{G})$, then $V \otimes \CC[\fg^*] \in \Rep_{a}(\widetilde{G})$. This defines an exact functor $\Rep_f(\widetilde{G}) \to \Rep_a(\widetilde{G})$, and hence a group homomorphism
    $$\phi_{\fg^*}: K_f(\widetilde{G}) \to K_a(\widetilde{G})$$
    \item[(vii)] The following diagram commutes:
    \begin{center}
        \begin{tikzcd}
        K_a(\widetilde{G}) & K_f(\widetilde{G}) \ar[l,"\phi_{\fg^*}"]\\
        K_{aa}(\widetilde{G}) \ar[u] & \\
        K^{\widetilde{G}}(\cN) \ar[u,"\Gamma(\bullet)|_{\widetilde{G}}"] \ar[r,"j_*"] & K^{\widetilde{G}}(\fg^*) \ar[uu,swap,"|_{\{0\}}"] \ar[uul,swap,"\Gamma(\bullet)|_{\widetilde{G}}"]
        \end{tikzcd}
    \end{center}
    \item[(viii)] The homomorphism in (v) is injective.
\end{itemize}
\end{prop}

\begin{proof}\leavevmode
\begin{itemize}
    \item[(i)] This is \cite[(5.4f)]{AdamsVogan} (it is an immediate consequence of the following facts: $\widetilde{G}$ is reductive, $\cN$ is affine, and $\widetilde{G}$ acts on $\cN$ with finitely many orbits).
    \item[(ii)] This is \cite[Cor 7.4]{AdamsVogan}.
    \item[(iii)] Since $\fg^*$ is affine, the functor $\Gamma(\bullet)|_{\widetilde{G}}: \Coh^{\widetilde{G}}(\fg^*) \to \Rep(\widetilde{G})$ is exact. It suffices to show that its image is contained in $\Rep_a(\widetilde{G})$. Let $\mathcal{E} \in \Coh^{\widetilde{G}}(\fg^*)$. Since $\fg^*$ is smooth, there is a $\widetilde{G}$-equivariant vector bundle $\mathcal{V}$ on $\fg^*$ and a surjection $\mathcal{V} \twoheadrightarrow \mathcal{E}$ in $\Coh^{\widetilde{G}}(\fg^*)$. Since $\Gamma(\bullet)|_{\widetilde{G}}$ is exact, we get a surjection $\Gamma(\fg^*,\mathcal{V})|_{\widetilde{G}} \twoheadrightarrow \Gamma(\fg^*,\mathcal{E})|_{\widetilde{G}}$ in $\Rep(\widetilde{G})$. So it suffices to show that $\Gamma(\fg^*,\mathcal{V})|_{\widetilde{G}} \in \Rep_a(\widetilde{G})$. Let $V = \mathcal{V}|_{\{0\}}$. Then $V \in \Rep_f(\widetilde{G})$ and
    $$\mathcal{V} \simeq V \otimes \cO_{\fg^*}$$
    in $\Coh^{\widetilde{G}}(\fg^*)$. So $\Gamma(\fg^*,\mathcal{V})|_{\widetilde{G}} = V \otimes \CC[\fg^*]|_{\widetilde{G}}$. Note that each graded component of $\CC[\fg^*]$, and hence of $V \otimes \CC[\fg^*]$, is finite-dimensional. So $V \otimes \CC[\fg^*]|_{\widetilde{G}} \in \Rep_a(\widetilde{G})$, as required.
    \item[(iv)] This is a well-known fact from equivariant $K$-theory, see \cite[Thm 4.1]{Thomason}. 
    \item[(v)] This follows from the fact that $j$ is closed, and hence affine.
    \item[(vi)] See the proof of (iii).
    \item[(vii)] It suffices to show that the top triangle is commutative (the commutativity of the bottom triangle is obvious). If $\mathcal{V} \simeq V \otimes \cO_{\fg^*} \in \Coh^{\widetilde{G}}(\fg^*)$ is a vector bundle, then
    $$\Gamma(\fg^*,\mathcal{V})|_{\widetilde{G}} \simeq V \otimes \CC[\fg^*]|_{\widetilde{G}} \simeq \mathcal{V}|_{\{0\}} \otimes \CC[\fg^*]|_{\widetilde{G}}$$
    But since $\fg^*$ is smooth, $K^{\widetilde{G}}(\fg^*)$ is spanned by vector bundles. So the upper triangle is commutative.
    \item[(viii)] By (vii), the map 
    \begin{equation}\label{eq:comp1}K^{\widetilde{G}}(\cN) \overset{\Gamma(\bullet)|_{\widetilde{G}}}{\to} K_{aa}(\widetilde{G}) \subset K_a(\widetilde{G})\end{equation}
    coincides with the composition
    $$K^{\widetilde{G}}(\cN) \overset{j_*}{\to} K^{\widetilde{G}}(\fg^*) \overset{\Gamma(\bullet)|_{\widetilde{G}}}{\to} K_a(\widetilde{G}).$$
    By (ii), (\ref{eq:comp1}) is injective. Hence, $j_*:K^{\widetilde{G}}(\cN) \to K^{\widetilde{G}}(\fg^*)$ must be injective as well.
\end{itemize}
\end{proof}

\begin{rmk}
We note that (ii) of Proposition \ref{prop:g*facts} (and hence (viii), which is a consequence) is a very deep assertion---the proof of (ii) in \cite{AdamsVogan} makes essential use of the Langlands classification.
\end{rmk}

\begin{rmk}
It is worth considering what happens if we forget about the $\CC^{\times}$-actions. Arguing exactly as in Proposition \ref{prop:g*facts}, we get a commutative diagram
    \begin{center}
        \begin{tikzcd}
        K(G) & K_f(G) \ar[l,"\phi_{\fg^*}"]\\
        K_a(G) \ar[u] & \\
        K^G(\cN) \ar[u,"\Gamma(\bullet)|_G"] \ar[r,"j_*"] & K^G(\fg^*) \ar[uu,swap,"|_{\{0\}}"] \ar[uul,swap,"\Gamma(\bullet)|_G"]
        \end{tikzcd}
    \end{center}
and $\Gamma: K^G(\cN) \to K_a(G)$ is injective. However, $K(G) = 0$ (indeed, every algebraic $G$-representation $V$ satisfies $V \oplus V^{\infty} \simeq V^{\infty}$, and therefore has image $0$ in $K(G)$). So we cannot deduce that $j_*: K^G(\cN) \to K^G(\fg^*)$ is injective (and in fact, it is not: the skyscraper sheaf at $\{0\}$ (with trivial $G$-action) lies in the kernel of the map $j_*: K^G(\cN) \to K^G(\fg^*)$. So, the $\CC^{\times}$-actions are essential for the proposition above. 
\end{rmk}

If we replace $\widetilde{G}$ with $\widetilde{K}$, $\fg^*$ with $\fp^*$, and so on, we can prove a result which is completely analogous to Proposition \ref{prop:g*facts} (the only change in the proof is that in (iii) we use \cite[Cor 10.9]{AdamsVogan} instead of \cite[Cor 7.4]{AdamsVogan}).   

Let $i: \fp^* \hookrightarrow \fg^*$ be the inclusion. The restriction functor $i^*: \Coh^{\widetilde{G}}(\fg^*) \to \Coh^{\widetilde{K}}(\fp^*)$ is not exact in general (if $\mathcal{E} \in \Coh^{\widetilde{G}}(\fg^*)$, then $i^*\mathcal{E}$ corresponds to the $\CC[\fp^*]$-module $\CC[\fp^*] \otimes_{\CC[\fg^*]} \mathcal{E}$. The functor $\CC[\fp^*] \otimes_{\CC[\fg^*]} (\bullet)$ is only right exact). Write $L_ni^*$ for its higher derived functors (if $\mathcal{E} \in \Coh^{\widetilde{G}}(\fg^*)$, then $L_ni^*\mathcal{E}$ corresponds to the $\CC[\fp^*]$-module $\mathrm{Tor}_n^{\CC[\fg^*]}(\CC[\fp^*],\mathcal{E})$). Since $\fg^*$ is smooth, $L_ni^*\mathcal{E} = 0$ for $n$ very large (see e.g. \cite[Prop 5.1.28]{Chriss-Ginzburg}). So we can define a homomorphism
$$i^*: K^{\widetilde{G}}(\fg^*) \to K^{\widetilde{K}}(\fp^*), \qquad  i^*[\mathcal{E}] = \sum_{n=0}^{\infty} (-1)^n [L_ni^*\mathcal{E}].$$
If $[\mathcal{E}]$ is supported in $\cN$, then $i^*[\mathcal{E}]$ is supported in $\cN_{\theta}$. So $i^*: K^{\widetilde{G}}(\fg^*) \to K^{\widetilde{K}}(\fp^*)$ restricts to a (unique) homomorphism $i^*: K^{\widetilde{G}}(\cN) \to K^{\widetilde{K}}(\cN_{\theta})$
\begin{equation}\label{eq:restriction}
    \begin{tikzcd}
    K^{\widetilde{G}}(\cN) \ar[r,hookrightarrow,"j_*"] \ar[d,"i^*"] & K^{\widetilde{G}}(\fg^*) \ar[d,"i^*"]\\
    K^{\widetilde{K}}(\cN_{\theta}) \ar[r,hookrightarrow,"j_*"] & K^{\widetilde{K}}(\fp^*)
    \end{tikzcd}
\end{equation}
We can compute these restriction maps in terms of $\widetilde{K}$-representations. The key ingredient is a graded Koszul identity in $K_a(\widetilde{K})$. Consider the class
$$[\wedge(\fk)] := \sum_{n=0}^{\infty} (-1)^n [\wedge^n(\fk)] \in K_f(\widetilde{K})$$
Here, as usual, we put $\fk$ in degree $1$
\begin{lemma}\label{lem:Koszul}
There is an identity in $K_a(\widetilde{K})$
$$\CC[\fk^*]|_{\widetilde{K}} \otimes [\wedge(\fk)] = \mathrm{triv}$$
\end{lemma}

\begin{proof}
Consider the Koszul resolution of the trivial $\CC[\fk^*]$-module
\begin{equation}\label{eq:Koszul}0 \to \CC[\fk^*] \otimes \wedge^{\dim(\fk)}(\fk) \to ... \to  \CC[\fk^*] \otimes \wedge^1(\fk)\to  \CC[\fk^*] \otimes \wedge^0(\fk)\to \CC \to 0\end{equation}
We can regard each term as a representation of $\widetilde{K}$, and it is easy to check that the differentials are $\widetilde{K}$-equivariant. Restricting to $\widetilde{K}$ we get an exact sequence in $\Rep_a(\widetilde{K})$
\begin{equation*}0 \to \CC[\fk^*]|_{\widetilde{K}} \otimes \wedge^{\dim(\fk)}(\fk) \to ... \to  \CC[\fk^*]|_{\widetilde{K}} \otimes \wedge^1(\fk)\to  \CC[\fk^*]|_{\widetilde{K}} \otimes \wedge^0(\fk)\to \mathrm{triv} \to 0\end{equation*}
Now the required identity follows from the Euler-Poincare principle
$$\CC[\fk^*]|_{\widetilde{K}} \otimes [\wedge(\fk)] = \sum_n (-1)^n \CC[\fk^*]|_{\widetilde{K}} \otimes [\wedge^n(\fk)]  = \mathrm{triv}$$
\end{proof}

If $V \in \Rep_a(\widetilde{G})$, then $V|_{\widetilde{K}} \in \Rep_a(\widetilde{K})$. This  defines an exact functor $\Rep_a(\widetilde{G}) \to \Rep_a(\widetilde{K})$, and hence a group homomorphism
$$|_{\widetilde{K}}: K_a(\widetilde{G}) \to K_a(\widetilde{K})$$
Tensoring with the class $[\wedge(\fk)] \in K_f(\widetilde{K})$, we obtain a further homomorphism
$$r: K_a(\widetilde{G}) \to K_a(\widetilde{K}), \qquad r[V] = [V]|_{\widetilde{K}} \otimes [\wedge(\fk)]$$
\begin{lemma}\label{lem:diagram}
The following diagram is commutative
\begin{center}
    \begin{tikzcd}
    K_a(\widetilde{G}) \ar[d,"r"] & K_f(\widetilde{G}) \ar[d,"|_{\widetilde{K}}"] \ar[l,"\phi_{\fg^*}"]\\
    K_a(\widetilde{K}) & K_f(\widetilde{K}) \ar[l,"\phi_{\fp^*}"]
    \end{tikzcd}
\end{center}
\end{lemma}

\begin{proof}
Let $V \in K_f(\widetilde{G})$. Then by Lemma \ref{lem:Koszul} we have
\begin{align*}
\phi_{\fg^*}(V)|_{\widetilde{K}} \otimes [\wedge(\fk)] &= (V \otimes \CC[\fg^*])|_{\widetilde{K}} \otimes [\wedge(\fk)]\\
&= V|_{\widetilde{K}} \otimes \CC[\fp^*]|_{\widetilde{K}} \otimes \CC[\fk^*]|_{\widetilde{K}} \otimes [\wedge(\fk)]\\
&= V|_{\widetilde{K}} \otimes \CC[\fp^*]|_{\widetilde{K}}\\
&= \phi_{\fp^*}(V|_{\widetilde{K}})
\end{align*}
as desired.
\end{proof}

\begin{lemma}\label{lem:diagram2}
The following diagram is commutative
\begin{center}
    \begin{tikzcd}
    K^{\widetilde{G}}(\fg^*) \ar[d,"i^*"] \ar[r,"|_{\{0\}}"] & K_f(\widetilde{G}) \ar[d,"|_{\widetilde{K}}"]\\
    K^{\widetilde{K}}(\fp^*) \ar[r,"|_{\{0\}}"] & K_f(\widetilde{K})
    \end{tikzcd}
\end{center}
\end{lemma}

\begin{proof}
If $\mathcal{V} \simeq V \otimes \cO_{\fg^*} \in \Coh^{\widetilde{G}}(\fg^*)$ is a vector bundle, then $L_ni^*\mathcal{V} = 0$ for $n>0$ ($\mathcal{V}$ corresponds to a flat, and hence projective, $\CC[\fg^*]$-module, so all higher Tor groups vanish). Consequently
\begin{align*}
(i^*[\mathcal{V}])|_{\{0\}} &= [i^*\mathcal{V}]|_{\{0\}}\\
&= (V|_{\widetilde{K}} \otimes \cO_{\fp^*})|_{\{0\}}\\
&= V|_{\widetilde{K}}\\
&= ([\mathcal{V}]|_{\{0\}})|_{\widetilde{K}}
\end{align*}
So the diagram commutes on vector bundles. But since $\fg^*$ is smooth, $K^{\widetilde{G}}(\fg^*)$ is spanned by vector bundles. This completes the proof.
\end{proof}

The next result gives a method for computing $i^*: K^{\widetilde{G}}(\cN) \to K^{\widetilde{K}}(\cN_{\theta})$ on the level of $\widetilde{K}$-representations.

\begin{cor}\label{cor:main}
The following diagram is commutative
\begin{center}
\begin{tikzcd}
 & K_a(\widetilde{G}) \ar[pos=.3,"r"]{dd} & &   K_f(\widetilde{G}) \ar[ll,hookrightarrow,"\phi_{\fg^*}"]\ar["|_{\widetilde{K}}"]{dd} \\
 
    K^{\widetilde{G}}(\cN)  \ar[crossing over,hookrightarrow,pos=.7,"j_*"]{rr} \ar[pos=.3,"i^*"]{dd} \ar[ur,hookrightarrow,"\Gamma"] & & K^{\widetilde{G}}(\fg^*)  \ar[ur,swap,pos=.3,"|_{\{0\}}"]\\
    
      & K_a(\widetilde{K}) & &  K_f(\widetilde{K})  \ar[ll,hookrightarrow,pos=.7,"\phi_{\fp^*}"] \\
    K^{\widetilde{K}}(\cN_{\theta}) \ar[ur,hookrightarrow,"\Gamma"] \ar[rr,"j_*",hookrightarrow] && K^{\widetilde{K}}(\fp^*) \ar[from=uu,crossing over,pos=.3,"i^*"] \ar[ur,swap,pos=.3,"|_{\{0\}}"]
\end{tikzcd}
\end{center}
In particular, for every $[\mathcal{E}] \in K^{\widetilde{G}}(\cN)$, the restriction $i^*[\mathcal{E}]$ is uniquely determined by the following identity in $K_a(\widetilde{K})$
\begin{equation}\label{eq:mainequation}\Gamma(i^*[\mathcal{E}])|_{\widetilde{K}} = \Gamma([\mathcal{E}])|_{\widetilde{K}} \otimes [\wedge(\fk)]\end{equation}
\end{cor}

\begin{proof}
The front face is commutative by the definition of $i^*: K^{\widetilde{G}}(\cN) \to K^{\widetilde{K}}(\cN_{\theta})$, see (\ref{eq:restriction}). The right face is commutative by Lemma \ref{lem:diagram2}. The back face is commutative by Lemma \ref{lem:diagram}. The top face is commutative by Proposition \ref{prop:g*facts}. The bottom face is commutative by its analog for $\widetilde{K}$. The commutativity of the left face follows as a formal consequence of the commutativity of the others.  
\end{proof}

\begin{rmk}\label{rmk:GGR}
There are surjective `forgetful' maps
$$K^{\widetilde{G}}(\cN) \twoheadrightarrow K^G(\cN), \qquad K^{\widetilde{K}}(\cN_{\theta}) \twoheadrightarrow K^K(\cN_{\theta})$$
It is not hard to show that the restriction map $i^*: K^{\widetilde{G}}(\cN) \to K^{\widetilde{K}}(\cN_{\theta})$ descends to a (necessarily unique) homomorphism
$$i^*: K^G(\cN) \to K^K(\cN_{\theta})$$
Since we will not use this fact, we will not prove it here. Write  $K\mathcal{M}(G_{\RR})$ for the Grothendieck group of finite-length admissible $G_{\RR}$-representations. Similarly, write $K\mathcal{M}(G)$ (here $G$ is regarded as a real reductive group by restriction of scalars). There are `associated graded' maps
$$\gr: K\mathcal{M}(G) \to K^G(\cN), \qquad \gr: K\mathcal{M}(G_{\RR}) \to K^K(\cN_{\theta})$$
(see \cite{Vogan1991} for definitions). An intriguing question, which we will not pursue in this paper, is whether there is a natural homomorphism $K\mathcal{M}(G) \overset{?}{\to} K\mathcal{M}(G_{\RR})$ such that the following diagram commutes
\begin{center}
    \begin{tikzcd}
    K\mathcal{M}(G) \ar[r,"?"] \ar[d,"\gr"] & K\mathcal{M}(G_{\RR}) \ar[d,"\gr"]\\
    K^G(\cN) \ar[r,"i^*"]& K^K(\cN_{\theta})
    \end{tikzcd}
\end{center}
%
\end{rmk}

\section{Regular functions on $\cN$}\label{sec:complexcase}

In this section, we will recall a (well-known) formula for $\CC[\cN]$ as a representation of $\widetilde{G}$. Choose a maximal torus $H \subset G$ and a Borel subgroup $B \subset G$ containing $H$. Let $\Lambda \subset \mathfrak{h}^*$ denote the weight lattice, $\Phi^+ \subset \Lambda$ the positive roots, and $\Lambda^+ \subset \Lambda$ the dominant weights. Then
$$\mathrm{Irr}(G) = \{\tau_{\lambda} \mid \lambda \in \Lambda^+\},$$
where $\tau_{\lambda}$ is the irreducible representation of $G$ with highest weight $\lambda$. Recall \emph{Kostant's partition function}
$$\mathcal{P}: \Lambda \to \ZZ, \qquad \mathcal{P}(\lambda) = \#\{m: \Phi^+ \to \ZZ \mid \lambda = \sum_{\alpha \in \Phi^+} m(\alpha)\alpha\}$$
Define
$$\mathcal{M}: \Lambda^+ \times \Lambda \to \ZZ, \qquad \mathcal{M}(\lambda,\mu) = \sum_{w \in W} (-1)^{\ell(w)} \mathcal{P}(w(\lambda+\rho)-(\mu+\rho))$$
where $\ell: W \to \ZZ_{\geq 0}$ is the length function and $\rho = \frac{1}{2}\sum \Phi^+$. For each $\lambda \in \Lambda^+$, there is an identity in $K_a(H)$
\begin{equation}\label{eq:Weyl}\tau_{\lambda}|_H = \sum_{\mu \in \Lambda} \mathcal{M}(\lambda,\mu)e^{\mu}\end{equation}
This is a version of the Weyl character formula.

Lusztig has introduced $q$-analogs of both $\mathcal{P}$ and $\mathcal{M}$ (\cite{Lusztig1983}). The $q$-analog of $\mathcal{P}$ is defined by the formula
$$\mathcal{P}_q: \Lambda \to \ZZ[q], \qquad \mathcal{P}_q(\lambda) = \sum_{n=0}^{\infty} \mathcal{P}_q^n(\lambda)q^n$$
where
$$\mathcal{P}_q^n(\lambda) = \#\{m: \Phi^+ \to \ZZ \mid \lambda = \sum_{\alpha \in \Phi^+} m(\alpha)\alpha \ \text{ and } \ n=\sum_{\alpha \in \Phi^+} m(\alpha)\}$$
The $q$-analog of $\mathcal{M}$ is
$$\mathcal{M}_q: \Lambda^+ \times \Lambda \to \ZZ[q], \qquad \mathcal{M}_q(\lambda,\mu) =  \sum_{w \in W}(-1)^{\epsilon(w)} \mathcal{P}_q(w(\lambda+\rho)-(\mu+\rho))$$
The following can be extracted from \cite{Mcgovern1989}.

\begin{prop}\label{prop:complexcase}
There is an identity in $K_a(\widetilde{G})$
$$\CC[\cN]|_{\widetilde{G}} = \sum_{\lambda \in \Lambda^+} \tau_{\lambda}\mathcal{M}_q(\lambda,0)$$
\end{prop}

\begin{proof}[Sketch of proof]
Consider the Springer resolution $\eta:T^*(G/B) \to \cN$. Since $T^*(G/B)$ is symplectic, there is an identification (of $\widetilde{G}$-equivariant sheaves)
$$\cO_{T^*(G/B)} \simeq \omega_{T^*(G/B)}$$
where $\omega_{T^*(G/B)}$ is the canonical sheaf on $T^*(G/B)$. So by the theorem of Grauert and Riemenschneider (\cite{GrauertRiemenschneider})
$$R^i\eta_*\mathcal{O}_{T^*(G/B)} =0, \qquad \forall i >0$$
On the other hand, $R^0\eta_*\cO_{T^*(G/B)} \simeq \cO_{\cN}$. Using the Leray spectral sequence (and the fact that $\cN$ is affine), we get an identity in $K_a(\widetilde{G})$
\begin{equation}\label{eq:CN}\CC[\cN]|_{\widetilde{G}} = \sum_i (-1)^i H^i(T^*(G/B),\cO_{T^*(G/B)})\end{equation}
If $p: T^*(G/B) \to G/B$ is the projection, then $p_*\cO_{T^*(G/B)}$ is identified (as a $\widetilde{G}$-equivariant sheaf) with (the sheaf of local sections of) the $\widetilde{G}$-equivariant vector bundle $G \times_B S(\mathfrak{g}/\mathfrak{b})$. Since $p$ is affine, the direct image functor $p_*$ preserves cohomology, i.e.
\begin{equation}\label{eq:applyp}H^i(T^*(G/B),\cO_{T^*(G/B)}) \simeq H^i(G/B,p_*\cO_{T^*(G/B)}), \qquad \forall i \geq 0\end{equation}
Combining (\ref{eq:CN}) and (\ref{eq:applyp}), we get a further identity in $K_a(\widetilde{G})$
$$\CC[\cN]|_{\widetilde{G}} = \sum_i (-1)^i H^i(G/B, G \times_B S(\mathfrak{g}/\mathfrak{b}))$$
The right hand side can be computed using Borel-Weil-Bott. The result follows.
\end{proof}

\section{Some commutative algebra}\label{sec:commalg}

Using the results of Section \ref{sec:Ktheory}, we get a well-defined class $i^*[\cO_{\cN}] \in K^{\widetilde{K}}(\cN_{\theta})$, which can be regarded as the restriction of $\cO_{\cN}$ to $\fp^*$. However, it is not at all clear that $i^*[\cO_{\cN}] = [\cO_{\cN_{\theta}}]$. For groups split modulo center, we will see that this equality always holds, but the proof will require some commutative algebra.

\begin{lemma}\label{lem:Tor}
Let $R$ be a Noetherian ring and let $M$ be a finitely-generated $R$-module. Suppose $x_1,...,x_m \in R$ is an $R$-regular sequence which is also $M$-regular. Consider the ideal $I = (x_1,...,x_m) \subset R$. Then
$$\mathrm{Tor}^R_n(R/I,M) = 0, \qquad n >0$$
\end{lemma}

\begin{proof}
Since $x_1,...,x_m$ is $R$-regular, the Koszul complex $K(x_1,...,x_m;R)$ is a resolution of $R/I$. So $\mathrm{Tor}_{\bullet}^R(R/I,M)$ is the homology of $K(x_1,...,x_m;M) := K(x_1,...,x_m;R) \otimes_R M$ 
$$H_n(K(x_1,...,x_m;M)) \simeq \mathrm{Tor}_n^R(R/I,M), \qquad \forall n$$
But since $x_1,...,x_m$ is an $M$-regular sequence, the complex $K(x_1,...,x_m;M)$ is acyclic, see \cite[Thm 16.5(i)]{matsumura_1987}. So $\mathrm{Tor}_n^R(R/I,M)=0$ for $n>0$.
\end{proof}

\begin{lemma}[Thm 17.4, \cite{matsumura_1987}]\label{lem:CM}
Suppose $A$ is Cohen-Macaulay, and let $I=(x_1,...,x_n) \subset A$ be an ideal. If
$$\dim(A/I) = \dim(A) - n$$
then $x_1,...,x_n$ is an $A$-regular sequence.
\end{lemma}

\begin{prop}\label{prop:commalg}
Suppose $X$ is a smooth Noetherian scheme and let $Y$,  $Z$ be closed subschemes of $X$. Write $i: Y \hookrightarrow X$ and $j: Z \hookrightarrow X$ for the inclusions and form the Cartesian diagram of schemes
\begin{center}
    \begin{tikzcd}
    Z \ar[r,hookrightarrow, "j"] & X\\
    Z \cap Y \ar[r,hookrightarrow] \ar[u,hookrightarrow] & Y \ar[u,hookrightarrow,"i"]
    \end{tikzcd}
\end{center}
Assume
\begin{itemize}
    \item[(i)] $Y$ is smooth.
    \item[(ii)] $Z$ is Cohen-Macaulay.
    \item[(iii)] $\dim(Z \cap Y) = \dim(Z) + \dim(Y) -\dim(X)$.
\end{itemize}
Then
\begin{equation}\label{eq:Ln}L_ni^*(j_*\mathcal{O}_Z) =0, \qquad n >0\end{equation}
\end{prop}

\begin{proof}
The statement is local in $X$, so we can assume all schemes are affine. Let $X=\mathrm{Spec}(R)$, $Y=\mathrm{Spec}(R/I)$ and $Z=\mathrm{Spec}(A)$, so that $Z\cap Y = \mathrm{Spec}(A/I)$. Since $Y$ is smooth, we can find an $R$-regular sequence $x_1,...,x_m\in R$ such that $I=(x_1,...,x_m)$, where $m=\dim(X)-\dim(Y)$. Now by Lemma \ref{lem:CM}, $(x_1,...,x_m)$ is an $A$-regular sequence. So by Lemma \ref{lem:Tor} (applied to the $R$-module $M=A$)
$$\mathrm{Tor}_n^R(R/I,A) =0, \qquad n>0$$
This is equivalent to (\ref{eq:Ln}). 
\end{proof}

\section{Regular functions on $\cN_{\theta}$}\label{sec:splitcase}

First, we will assume that $G_{\RR}$ is split. This means that $G$ contains a maximal torus $H_s \subset G$ on which $\theta$ acts by inversion. Our first lemma shows that the codimension of $\cN_{\theta} \subset \cN$ equals the codimension of $\fp^* \subset \fg^*$.

\begin{lemma}\label{lem:splitdim}
$\dim(\cN_{\theta}) = \dim(\cN) + \dim(\fp) - \dim(\fg)$
\end{lemma}

\begin{proof}
By \cite[Prop 9]{KostantRallis1971}
\begin{equation}\label{eq:dimformula1}\dim(\cN_{\theta}) = \dim(\fp)-\dim(\mathfrak{h}_s)\end{equation}
By the Iwasawa decomposition
\begin{equation}\label{eq:dimformula2}\dim(\fg)  = \dim(\fk) + \dim(\mathfrak{h}_s) + \dim(\mathfrak{n})
\end{equation}
Finally
\begin{equation}\label{eq:dimformula3}\dim(\cN) = 2\dim(\mathfrak{n}) = \dim(\mathfrak{g}) - \dim(\mathfrak{h}_s)\end{equation}
Combining (\ref{eq:dimformula1}), (\ref{eq:dimformula2}), and (\ref{eq:dimformula3}) proves the lemma.
\end{proof}

\begin{cor}\label{cor:formulasplit}
There is an equality in $K^{\widetilde{K}}(\cN_{\theta})$
\begin{equation}\label{eq:split1}[\cO_{\cN_{\theta}}] = i^*[\cO_{\cN}]\end{equation}
and hence an equality in $K_{aa}(\widetilde{K})$
\begin{equation}\label{eq:split2}\CC[\cN_{\theta}]|_{\widetilde{K}} = \CC[\cN]|_{\widetilde{K}} \otimes [\wedge(\fk)]\end{equation}
\end{cor}

\begin{proof}
For (\ref{eq:split1}), we will apply Proposition \ref{prop:commalg}. So let $X=\fg^*$, $Z = \cN$, and $Y=\fp^*$. Clearly, $X$ and $Y$ are smooth. By \cite[Thm 0.1]{Kostant1963}, $\cN$ is a complete intersection, and therefore Cohen-Macaulay. Condition (iii) of Proposition \ref{prop:commalg} is the content of Lemma \ref{lem:splitdim}. So by Proposition \ref{prop:commalg}, we have
$$L_ni^*(j_*\cO_{\cN}) =0, \qquad n>0$$
and therefore
$$i^*[\cO_{\cN}] = \sum_n (-1)^n [L_ni^*j_*\cO_{\cN}] = [i^*j_*\cO_{\cN}] = [\cO_{\cN_{\theta}}]$$
For the final equality, we use the well-known fact that scheme-theoretic intersection $\cN \cap \fp^*$ is reduced, see \cite[Theorem 14]{KostantRallis1971}. This proves (\ref{eq:split1}). Now (\ref{eq:split2}) follows from (\ref{eq:split1}) and Corollary \ref{cor:main}.
\end{proof}

Corollary \ref{cor:formulasplit} can be easily extended to the case when $G_{\RR}$ is split modulo center. 

\begin{example}
Let $G = \mathrm{SL}_2(\CC)$ and let $\theta(g) = (g^{-1})^t$ (this is the involution corresponding to split real form $\mathrm{SL}_2(\RR)$). Then $K = \mathrm{SO}_2(\CC)$. Write
$$\mathrm{Irr}(G) = \{\tau_m \mid m = 0,1,2,...\}, \qquad \mathrm{Irr}(K) = \{\chi_n \mid n \in \ZZ\}$$
(here $\tau_m$ is the irreducible with highest weight $m$ and $\chi_n$ is the degree-$n$ character of $\mathrm{SO}_2(\CC)$). We have the following branching rules
$$\tau_m|_K = \chi_{-2m} + \chi_{-2m+2} + ... + \chi_{2m}.$$
From Proposition \ref{prop:complexcase}, we deduce
$$\CC[\cN]|_{\widetilde{G}} = \sum_{m=0}^{\infty} \tau_{2m}q^m$$
Also
$$[\wedge(\fk)] = \chi_0 - \chi_0q.$$
So by Corollary \ref{cor:formulasplit}
\begin{align*}
\CC[\cN_{\theta}]|_{\widetilde{K}} &= \CC[\cN]|_{\widetilde{K}} \otimes [\wedge(\fk)]\\
&= (\sum_{m=0}^{\infty} (\chi_{-2m}... + \chi_{2m})q^m) \otimes (\chi_0 - \chi_0q)\\
&= \sum_{m=0}^{\infty} ((\chi_{-2m}+...+\chi_{2m}) - (\chi_{-2m+2} +...+\chi_{2m-2}))q^m\\
&= \sum_{m=0}^{\infty} (\chi_{2m} +\chi_{-2m})q^m
\end{align*}
\end{example}

\section{Branching to $K$}\label{sec:mainresults}

In this section, we will use Corollary \ref{cor:formulasplit} to compute $\CC[\cN_{\theta}]|_{\widetilde{K}}$ as a formal integer combination of classes of the form $I(\Gamma)q^n$. 

Suppose $(H,\gamma,\Phi^+)$ is a continued Langlands parameter (see Definition \ref{def:Langlands}) and let $\chi \in K_f(K)$. It is easy to compute the tensor product $[I(H,\gamma,\Phi^+)]|_K \otimes \chi$ as a representation of $K$.

\begin{lemma}[Lem 12.13, \cite{Vogan2007}]\label{lem:tensorstandard}
Choose a finite multiset $S_H(\chi)$ in $X^*(H)$ such that
$$\chi|_{H^{\theta}} = \sum_{\mu \in S_H(\chi)} \mu|_{H^{\theta}}$$
Then there is an identity in $K_a(K)$
$$[I(H,\gamma,\Phi^+)]|_K \otimes \chi = \sum_{\mu \in S_H(\chi)} [I(H,\gamma+\mu,\Phi^+)]$$
\end{lemma}

We will use Lemma \ref{lem:tensorstandard} (together with Zuckerman's character formula for the trivial representation) to compute $\CC[\cN_{\theta}]|_{\widetilde{K}}$ in terms of the classes $I(\Gamma)q^n$. For an arbitrary class $\chi \in K_f(K)$, it may be difficult to find a multiset $S_H(\chi)$ as in Lemma \ref{lem:tensorstandard}. Fortunately, in our setting, $\chi$ is \emph{not} arbitrary. For the problem at hand, we will need to compute $S_H(\chi)$ in the following two cases: 
\begin{enumerate}
    \item $\chi$ is the restriction to $K$ of an irreducible representation $\tau_{\lambda}$ of $G$.
    \item $\chi$ is the class $\wedge^n(\fk)$.
\end{enumerate}

First, suppose $\chi = \tau_{\lambda}|_K$. By the Weyl character formula (\ref{eq:Weyl}), we have
$$\tau_{\lambda}|_H = \sum_{\mu \in X^*(H)} \mathcal{M}(\lambda,\mu)e^{\mu}$$
So we can take 
$$S_H(\tau_{\lambda}) = \{\mathcal{M}(\lambda,\mu)e^{\mu} \mid \mu \in X^*(H)\}$$
(the coefficients $\mathcal{M}(\lambda,\mu)$ above denote the multiset multiplicities. We will use similar notation below).

Next, suppose $\chi = \fk$. Choose a subset $\Delta_{\CC}' \subset \Delta_{\CC}$ such that for each $\alpha \in \Delta_{\CC}$, exactly one of $\{\alpha,\theta \alpha\}$ appears in $\Delta'_{\CC}$. Then there is a decomposition of $\fk$ into weight spaces for $H^{\theta}$
$$\fk \simeq \bigoplus_{\alpha \in \Delta_c} \fg_{\alpha} \oplus \bigoplus_{\alpha \in \Delta_{\CC}'} (1+d\theta)\fg_{\alpha} \oplus \bigoplus_{\alpha \in \Delta^+_{\RR}} (1+d\theta) \fg_{\alpha}$$
So we can take 
$$S_H(\fk) = \{e^{\alpha} \mid \alpha \in \Delta_c\} \cup \{e^{\alpha} \mid \alpha \in \Delta_{\CC}'\} \cup \{|\Delta_{\RR}^+|e^0\}$$
More generally
$$S_H(\wedge^n\fk) = \{\sum R \mid R \subseteq S_H(\mathfrak{k}), \ |R| = n\}$$
Zuckerman's character formula for the trivial representation (see \cite[Thm 9.4.16]{Vogan1981}) can be interpreted as an identity in $K_a(K)$
$$\mathrm{triv} = \sum_H \sum_{\Delta^+} (-1)^{\ell(\Delta^+)} [I(H,\rho_{i\RR},\Delta_{i\RR}^+)]|_K$$
The outer sum runs over $K$-conjugacy classes of $\theta$-stable maximal tori $H \subset G$ and the inner sum over $W(K,H^{\theta})$-conjugacy classes of positive systems $\Delta^+ \subset \Delta(G,H)$. The integer $\ell(\Delta^+)$ is the codimension of the $K$-orbit on the flag variety containing the Borel subalgebra $\mathfrak{b} \subset \fg$ corresponding to the positive system $\Delta^+$. Now assume that $G_{\RR}$ is split modulo center. Using Corollary \ref{cor:formulasplit}, we obtain an identity in $K_a(\widetilde{K})$
$$\CC[\cN_{\theta}]|_{\widetilde{K}} = \left(\sum_H \sum_{\Delta^+} (-1)^{\ell(\Delta^+)} [I(H,\rho_{i\RR},\Delta_{i\RR}^+)]|_K\right) \otimes \CC[\cN]|_{\widetilde{K}} \otimes [\wedge(\fk)]$$
Using Proposition \ref{prop:complexcase} and Lemma \ref{lem:tensorstandard}, we can rewrite the right hand side in terms of classes of the form $[I(\Gamma)]q^n$

\begin{theorem}\label{thm:branching}
Assume $G_{\RR}$ is split modulo center. Then there is an identity in $K_a(\widetilde{K})$
$$\CC[\cN_{\theta}]|_{\widetilde{K}} =  \sum_H \sum_{\Delta^+} \sum_{\lambda \in \Lambda^+} \sum_{\mu \in \Lambda} \sum_{R \subseteq S_H(\fk)} (-1)^{\ell(\Delta^+)+|R|} \mathcal{M}(\lambda,\mu) \mathcal{M}_q(\lambda,0) [I(H,\rho_{i\RR} + \mu + |R|,\Delta_{i\RR}^+)]q^{|R|}
$$
\end{theorem}
Note that the terms on the right are not final. We can rewrite the sum in terms of final parameters using the Hecht-Schmid identities (this sort of thing is easy to do in atlas).

\begin{sloppypar} \printbibliography[title={References}] \end{sloppypar}

\end{document}